\documentclass{amsart}
\usepackage{amsmath} 
\usepackage{amssymb}
\usepackage{mathrsfs}

\newtheorem{theorem}{Theorem}[section] 
\newtheorem{claim}[theorem]{Claim}
\newtheorem{lemma}[theorem]{Lemma} 
 
\newtheorem{conclusion}[theorem]{Conclusion}
\newtheorem{observation}[theorem]{Observation}

\theoremstyle{definition}
\newtheorem{definition}[theorem]{Definition}

\newtheorem{discussion}[theorem]{Discussion}

\theoremstyle{remark}
\newtheorem{remark}[theorem]{Remark}
\newtheorem{question}[theorem]{Question}
\newtheorem{notation}[theorem]{Notation}

\newcommand{\pp}{{\rm pp}}
\newcommand{\BB}{{\rm BB}}
\newcommand{\nst}{{\rm nst}}
\newcommand{\pcf}{{\rm pcf}}
\newcommand{\Sol}{{\rm Sol}}
\newcommand{\Sep}{{\rm Sep}}
\newcommand{\Reg}{{\rm Reg}}

\newcommand{\OBB}{{\rm OBB}}

\newcommand{\otp}{{\rm otp}}

\newcommand{\bd}{{\rm bd}}

\newcommand{\tcf}{{\rm tcf}}

\newcommand{\Dom}{{\rm Dom}}

\newcommand{\rest}{{\restriction}}

\newcommand{\then}{{\underline{then}}}
\newcommand{\when}{{\underline{when}}}
\newcommand{\Then}{{\underline{Then}}}

\newcommand{\mn}{{\medskip\noindent}}
\newcommand{\sn}{{\smallskip\noindent}}

\newcommand{\cA}{{\mathscr A}}

\newcommand{\gD}{{\mathfrak D}}

\newcommand{\cD}{{\mathscr D}}
\newcommand{\gd}{{\mathfrak d\/}}

\newcommand{\cF}{{\mathscr F}}

\newcommand{\cP}{{\mathscr P}}

\newcommand{\bbZ}{{\mathbb Z}}

\newcommand{\scr}{{\mathscr S}}

\newcommand{\gt}{{\mathfrak t}} 
\newcommand{\cU}{{\mathscr U}}

\newcommand{\cW}{{\mathscr W}}

\newcommand{\cf}{{\rm cf}}

\newcount\skewfactor
\def\mathunderaccent#1#2 {\let\theaccent#1\skewfactor#2
\mathpalette\putaccentunder}
\def\putaccentunder#1#2{\oalign{$#1#2$\crcr\hidewidth
\vbox to.2ex{\hbox{$#1\skew\skewfactor\theaccent{}$}\vss}\hidewidth}}

\newenvironment{PROOF}[2][\proofname.]
   {\begin{proof}[#1]\renewcommand{\qedsymbol}{\textsquare$_{\rm #2}$}}
   {\end{proof}}

\begin{document}

\title {Ordered black boxes: existence}
\author {Saharon Shelah}
\address{Einstein Institute of Mathematics\\
Edmond J. Safra Campus, Givat Ram\\
The Hebrew University of Jerusalem\\
Jerusalem, 91904, Israel\\
 and \\
 Department of Mathematics\\
 Hill Center - Busch Campus \\ 
 Rutgers, The State University of New Jersey \\
 110 Frelinghuysen Road \\
 Piscataway, NJ 08854-8019 USA}
\email{shelah@math.huji.ac.il}
\urladdr{http://shelah.logic.at}
\thanks{The author would like to thank the Israel Science Foundation
for partial support of this research (Grant No. 1053/11).
The author thanks Alice Leonhardt for the beautiful typing.
  First typed April 30, 2012 when moved from \cite{Sh:898} which was
  first typed Dec. 6, 2006.}





\date{December 16, 2012}

\begin{abstract}
We defined ordered black boxes in which for a partial $J$ we try to
predict just a bound in $J$ to a function restricted to $C_\alpha$.
The existence results are closely related to pcf, propagating
downward.  We can start with trivial cases.
\end{abstract}

\maketitle
\numberwithin{equation}{section}
\setcounter{section}{-1}
\newpage

\section{Introduction}

We defined the so-called ordered black box, we use pcf to get many
cases of it by downward induction.  The starting point is via $\lambda
= \cf(2^\mu),\mu \in \bold C_\kappa$.  Via such black boxes we get
sometime cases of black boxes which are enough, e.g. for the ``TDC =
the trivial dual conjecture".

This was part of \cite{Sh:898}, but in the end the referee requested it
to be moved out as it was not used.

We use some definitions from \cite{Sh:898}.
\newpage

\section{Propagating OBB$_\sigma(\bar C)$ down by pcf}

We deal here with the ordered black box, OBB and prove in ZFC that
many cases occur.
\begin{definition}
\label{j.35}
1) For a partial order $J$, a sequence $\bar
C = \langle C_\delta:\delta \in S\rangle$ such that $C_\delta
\subseteq \delta$ for every $\delta \in S$ and an ideal $I$ on $S$, let
$\OBB_J(\bar C,I)$ mean that there exists a sequence $\langle
t_\delta:\delta \in S \rangle$ with $t_\delta \in J$ 
such that: if $f:\bigcup\limits_{\delta \in S} C_\delta
\rightarrow J$, \then \, $\{\delta \in S:(\forall \alpha \in
C_\delta)(f(\alpha) \le_J t_\delta)\} \ne \emptyset \mod I$.

\noindent
2) If $\sup(S) = \sup(\bigcup\limits_{\delta \in S} C_\delta)$
is a regular uncountable cardinal and $I$ is the
 non-stationary ideal on $\lambda$ restricted to $S$, \then \, we may omit $I$.

\noindent
3) If $J = (\theta,<)$, we may write $\theta$ instead of $J$.

\noindent
4) $\OBB^+_J(\bar C,I)$ is defined as in part (1) but we demand $\{\delta
\in S:(\forall \alpha \in C_\delta)(f(\alpha) \le_J t_\delta)\} = S
\mod I$.
\end{definition}

\begin{remark}
Note that we can use only $I,J$ and $\lambda(\bar C)$, see below 
which is a regular cardinal.
\end{remark}

\begin{notation}
\label{j.41}
Let $\bar C = \langle C_\delta:\delta \in S \rangle,
S = S(\bar C)$ and $\Dom(\bar C) := \cup\{C_\delta:\delta \in S\}$.

\noindent
1) We may write
$\kappa = \kappa(\bar C)$ when $\delta \in S \Rightarrow 
\otp(C_\delta) = \kappa$.

\noindent
2) We may write
$\lambda = \lambda(\bar C)$ when $S \subseteq \lambda$ is a
stationary subset of the regular uncountable cardinal $\lambda$ 
and $[\delta \in S \Rightarrow C_\delta \subseteq \delta = \sup(C_\delta)]$.

\noindent
3) We may write $\mu = \mu(\bar C)$ 
when $\mu = \sup \cup \{C_\delta:\delta \in S\}$ is $< |S(\bar C)|$.

\noindent
4) We say $\bar C$ is tree-like when $\alpha \in C_{\delta_1} \cap
 C_{\delta_2} \Rightarrow C_{\delta_1} \cap \alpha = C_{\delta_2}
\cap \alpha$.
\end{notation}

\noindent
A (trivial) starting point is
\begin{observation}
\label{j.44}
1) If $\lambda = \lambda(\bar C)$ is
(well-defined and) regular uncountable and $I$ is the non-stationary
ideal restricted to $S$ (which is a stationary subset of $\lambda$),
see Definition \ref{j.41}(2), \then \,
$\OBB_\lambda(\bar C,I)$ and moreover $\OBB^+_\lambda(\bar C,I)$.

\noindent
2) If $\OBB_J(\bar C,I)$ as exemplified by $\bar t = \langle
t_\delta:\delta \in S\rangle$ and $\theta = \cf(\theta)$ and
$J$ is a $\theta$-directed partial order and $I^*_{\bar t} := \{\cup\{A_i:i
  < i(*)\}:i(*) < \theta$ and for each $i < \theta$ for some
  $f_i:\cup\{C_\delta:\delta \in S\} \rightarrow J$ we have $A_i
 = \{\delta \in S:\neg(\forall \alpha \in C_\delta)(f_i(\alpha)
   \le_J t_\delta)\}\}$ \then \, $\OBB^+_J(\bar C,I^*_{\bar t}))$
and $I^*_{\bar t}$ is a $\theta$-complete ideal on $\lambda(\bar C)$.

\noindent
3) If $\OBB^+_J(\bar C,I)$ and $S_1 \in I^+$ then $\OBB^+_J(\bar C,I +
(S \backslash S_1))$, in fact for $J,S,\bar C,I$ as in \ref{j.35},
$\OBB^+_J(\bar C,I)$ iff $\OBB_J(\bar C,I + (S \backslash S_1))$ for
 every $S_1 \in I^+$.
\end{observation}

\begin{PROOF}{\ref{j.44}}
1) Let $S = S(\bar C)$.  We define $\bar t = \langle
t_\delta:\delta \in S\rangle$ by letting $t_\delta = \delta$.

Now let $f:\cup\{C_\delta:\delta \in S\} \rightarrow \lambda$ but
 $\cup\{C_\delta:\delta \in S\} \subseteq \lambda$ so let $f^+
\supseteq f$ be such that
$f^+:\lambda \rightarrow \lambda$ (we usually below use
just any such $f^+$).

So $E = \{\delta < \lambda$ \,: for every $\alpha < \delta$ we have
$f^+(\alpha) < \delta\}$ is a club of $\lambda$ and for every 
$\delta \in S \cap E$
we have $\alpha \in C_\delta \Rightarrow \alpha < \delta \Rightarrow
f(\alpha) < \delta$.

As $S \cap E$ is stationary, we are done.

\noindent
2) Obviously, $I^*_{\bar t} \subseteq \cP(S)$ is closed under subsets, and as
   $\theta$ is infinite regular, clearly $I^*_{\bar t}$ 
is closed under union of $< \theta$ members, 
hence $\OBB^+_J(\bar C,I^*_{\bar t})$ holds by the
   definition \ref{j.35} provided that we show $S \notin I^*_{\bar
   t}$.  So assume $i(*) < \theta,f_i:\cup\{C_\delta:\delta \in S\}
   \rightarrow J$ and $A_i := \{\delta \in S \neg (\forall \alpha \in
   C_\alpha)(f_i(\alpha) \le_J t_\delta)\}$ and we should prove $S \ne
   \cup\{A_i:i < i(*)\}$.  Choose $f:\cup\{C_\delta:\delta \in S\}
   \rightarrow J$ by: $f(\alpha)$ is any $\le_J$-upper bound of
   $\{f_i(\alpha):i < i(*)\}$, exists as $J$ is $\theta$-directed.
Let $A := \{\delta \in S:(\forall \alpha \in
   C_\delta)(f(\alpha)) \le t_\delta\}$, so $A \in I^+$ as 
$\OBB_J(\bar C,I)$ is assumed, so necessarily $A \ne \emptyset$ by
the definition of $A$; as clearly $i<i(*) \Rightarrow A \subseteq
A_i$, so we are done.

\noindent
3) Easy.  
\end{PROOF}

\noindent
The main case in Lemma \ref{j.45} below is: 
each $J_i$ is a regular cardinal $> i^*$.
\begin{lemma}
\label{j.45}
If clauses (a)-(f) below hold, \then \, for
some ${\cU} \in I^+_*$, for every $i \in {\cU}$, we have
$\OBB_{J_i}(\bar C,I)$, where:
\mn
\begin{enumerate}
\item[$(a)$]  $\OBB_J(\bar C,I)$ such that
$S = S(\bar C),\kappa = \kappa(\bar C)$ or just $\delta \in S
\Rightarrow \otp(C_\delta) \le \kappa$
\sn
\item[$(b)$]   $I$ is an $|i^*|^+$-complete ideal on $S$
\sn
\item[$(c)$]  $\bar J = \langle J_i:i < i^*\rangle$
\sn
 \item[$(d)$]  $J_i$ is a partial order such that $J_i \models
(\forall s)(\exists t)(s < t)$
\sn
\item[$(e)$]  $I_*$ is an ideal on $i^*$
\sn
\item[$(f)$]  $(\alpha) \quad J$ is a partial order
\sn
\item[${{}}$]  $(\beta) \quad \bar g = \langle g_t:t \in J\rangle$
\sn
\item[${{}}$]  $(\gamma) \quad g_t \in  \prod_{i < i^*} J_i$ for $t
\in J$
\sn
\item[${{}}$]  $(\delta) \quad \{g_t:t \in J\}$ is cofinal in 
$J_* = (\prod\limits_{i <i^*} J_i,\le_{I_*})$, where the partial 

\hskip25pt order $\le_{I_*}$ on $\prod\limits_{i < i^*} J_i$ 
is defined by $g' \le_{I_*} g''$ iff 

\hskip25pt $\{i < i^*:g'(i) \le_{J_i} g'(i)\} = i^* \mod I_*$
\sn
\item[${{}}$]  $(\varepsilon) \quad \bar g$ is $\le_{I_*}$-increasing,
i.e., $s \le_J t \Rightarrow g_s \le_{I_*} g_t$,
\item[$(g)$]  one of the following\footnote{we may label this
$(f)(\zeta)$ rather than (g), but as it is 
much bigger we prefer the present form}  possibilities holds:
\mn
\begin{enumerate}
\item[${{}}$]  \underline{Pos (A)}:  $(\alpha) \quad 
I_*$ is $\kappa^+$-complete or just
\sn
\item[${{}}$]  \hskip45pt $(\alpha)' \quad I_*$ is $|C_\delta|^+$-complete for
every $\delta \in S$
\sn
\item[${{}}$]  \underline{Pos (B)}: $(\alpha) \quad \bar C$ is
tree-like and $\otp(C_\delta) = \kappa$ for $\delta \in S$
\sn
\item[${{}}$]  \hskip20pt $(\beta) \quad J_i$ is $\kappa$-directed
\sn
\item[${{}}$]  \hskip20pt $(\gamma) \quad$ if ${\cU}_\varepsilon 
\in I_*$ is $\subseteq$-increasing for $\varepsilon < \kappa$ then 

\hskip45pt $i^* \backslash \cup\{{\cU}_\varepsilon:
\varepsilon < \kappa\} \notin I_*$
\sn
\item[${{}}$]  \underline{Pos (C)}:  there is $\bold
F:\prod\limits_{i<i^*} J_i \rightarrow J$ satisfying $f \le_{I_*} 
g_{\bold F(f)}$ for every 

\hskip45pt  $f \in \prod\limits_{i < i^*} J_i$
 and if $f_\varepsilon \in \prod\limits_{i<i^*} J_i$ and 
$\bold F(f_\varepsilon) \le_J t$ for $\varepsilon < \kappa$,

\hskip45pt  then 
$\{i < i^*:f_\varepsilon(i) \le_{J_i} g_t(i)$ for every $\varepsilon <
\kappa\} \ne \emptyset \mod I_*$
\sn
\item[${{}}$]  \underline{Pos (D)}:  Clauses $(\alpha),(\beta)$ as
in Pos $(B)$, and
\sn
\item[${{}}$]  \hskip20pt $(\gamma) \quad$ if $r_\zeta \le_J t$ for
$\zeta < \kappa$, then $\{i < i^*:g_{r_\varepsilon}(i)
\le_{J_i} g_t(i)$ for every

\hskip45pt  $\varepsilon < \kappa\} \ne \emptyset \mod I_*$
\item[${{}}$]  \underline{Pos (E)}:  Clause $(\alpha)$ as in Pos (B)
and
\sn
\item[${{}}$]  \hskip20pt $(\beta) \quad 
\bold F_\varepsilon:{}^{\varepsilon +1}(\prod\limits_{i <i^*} J_i)
\rightarrow J$ for $\varepsilon < \kappa$ such that if $f_\zeta \in
\prod\limits_{i < i^*} J_i$ and

\hskip45pt  $t_\zeta = \bold F_\zeta(\langle
f_\varepsilon:\varepsilon \le \zeta\rangle)$ and $t_\zeta \le_J t$ for
$\zeta < \kappa$ \then \, 

\hskip45pt $\{i < i^*:(\forall \varepsilon <
\kappa)(f_\varepsilon(i) \le_{J_i} g_t(i)\} \ne \emptyset \mod I_*$

\hskip45pt and $f_\zeta \le_{J_*} g_{t_\zeta}$.
\end{enumerate}
\end{enumerate}
\end{lemma}

\begin{remark}
1) In Pos(A) of \ref{j.45}, when clause $(\alpha)$ holds 
we get ${\cU} = i^* \mod I$.
\end{remark}

\begin{PROOF}{\ref{j.45}}
Let $B = \cup\{C_\delta:\delta \in S(\bar C)\}$.

Let $\bar t = \langle t_\delta:\delta \in S\rangle$
witness $\OBB_J(\bar C,I)$.  For each $i < i^*$, we consider $\bar s^i :=
\langle g_{t_\delta}(i):\delta \in S\rangle \in {}^S(J_i)$.  
We denote the $\delta$-th member of $\bar s^i$ by $s^i_\delta$, so
$s^i_\delta = g_{t_\delta}(i)$ for $i<i^*$ and $\delta \in S$.

Let
\mn
\begin{enumerate}
\item[$(*)_1$]    ${\cU}_0 := \{i<i^*:\bar s^i$ is a witness
for $\OBB_{J_i}(\bar C,I)\}$.
\end{enumerate}
\mn
It suffices to prove
\mn
\begin{enumerate}
\item[$(*)_2$]  ${\cU}_0 \notin I_*$
\end{enumerate}
[Why?  Obviously.]

Now for each $i \in {\cU}_1 := i^* \backslash {\cU}_0$ 
let $f_i:B \rightarrow J_i$ exemplify that $\bar s^i$ is not a witness for
$\OBB_{J_i}(\bar C,I)$, i.e.,
\mn
\begin{enumerate}
\item[$(*)_3$]   if $i \in i_* \backslash {\cU}_0$ then $W_i =
\emptyset$ mod $I$ where $W_i := \{\delta \in S:(\forall \alpha \in
C_\delta)(f_i(\alpha) \le_{J_i} s^i_\delta)\}$.
\end{enumerate}
\mn
If $i \in {\cU}_0$, choose any $f_i:B \rightarrow J_i$.

Now
\mn
\begin{enumerate}
\item[$(*)_4$]  $W := \cup\{W_i:i \in {\cU}_1\} \in I$.
\end{enumerate}
\mn
[Why?  By clause (b) of the assumption, the ideal $I$ is $|i^*|^+$-complete.]

Now we choose for each $\alpha \in B$ a function $h_\alpha$ as
follows:
\mn
\begin{enumerate}
\item[$(*)_5$]  $(a) \quad h_\alpha \in \prod\limits_{i<i^*} J_i$
\sn
\item[${{}}$]   $(b) \quad i <i^* \Rightarrow f_i(\alpha) 
\le_{J_i} h_\alpha(i)$
\sn
\item[${{}}$]   $(c) \quad$  if ($\bar C$ is tree-like and each
$J_i$ is $\kappa$-directed),

\hskip25pt  \then \,
$h_\alpha(i)$ is a common $\le_{J_i}$-upper bound of $\{f_i(\alpha)\}
\cup \{h_\beta(i)$:

\hskip25pt  there is $\delta \in S$ such that
$\alpha \in C_\delta \wedge \beta \in C_\delta \cap \alpha\}$; 

\hskip25pt so $\langle h_\alpha(i):\alpha \in C_\delta\rangle$
is $\le_{J_i}$-increasing for each $\delta \in S$.
\end{enumerate}
\mn
[Why does such an $h_\alpha$ exist?  If the assumption of $(*)_5(c)$ 
fails, we let
$h_\alpha(i) = f_i(\alpha)$.  If the assumption of $(*)_5(c)$ holds, then (for
$\alpha \in B,i<i^*$) the set $\{f_\beta(i):\beta = \alpha$ or
$(\exists \delta \in S)(\alpha \in C_\delta \wedge \beta \in C_\delta
\cap \alpha)\}$ has cardinality $< \kappa$ because 
$\bar C$ is tree-like and has a common $\le_{J_i}$-upper bound since 
$J_i$ is $\kappa$-directed and let $h_\alpha(i)$ be any such bound.]

So for $\alpha \in B$, we have $h_\alpha \in \prod\limits_{i<i^*} J_i$, and
hence by clause $(f)(\delta)$ of the assumption, we can choose 
$s_\alpha \in J$ such that:
\mn
\begin{enumerate}
\item[$(*)_6$]  $(a) \quad h_\alpha \le_{I_*} g_{s_\alpha}$, i.e.,

$\quad {\cU}^0_\alpha := \{i < i^*:\neg(h_\alpha(i) \le_{J_i}
g_{s_\alpha}(i))\} \in I_*$
\sn
\item[${{}}$]  $(b) \quad$ in Pos(C), $s_\alpha := \bold F(h_\alpha)$
\sn
\item[${{}}$]  $(c) \quad$ in Pos(E), if $\alpha \in C$ and $\zeta
< \kappa$ letting 
$\langle \alpha_\varepsilon:\varepsilon \le \zeta\rangle$ list
$C_\delta \cap (\alpha + 1)$, 

\hskip20pt for any $\delta \in S$ such that $\alpha
\in C_\delta$, we have 

\hskip20pt $s_\alpha = \bold F_\zeta(\langle h_{\alpha_\varepsilon}:
\varepsilon \le \zeta\rangle)$.
\end{enumerate}
\mn
Note that the demands in clauses (b),(c) of $(*)_6$ are compatible with
the demand in clause (a) of $(*)_6$.

So $\alpha \mapsto s_\alpha$ is a function from $B = \bigcup\limits_{\delta \in
S} C_\delta$ to $J$, but $\langle t_\delta:\delta \in S\rangle$ was
chosen exemplifying $\OBB_J(\bar C,I)$, hence
\mn
\begin{enumerate}
\item[$(*)_7$]  $W_* := \{\delta \in S$ \,: if $\alpha \in C_\delta$, 
then $s_\alpha \le_J t_\delta\}$ belongs to $I^+$.
\end{enumerate}
\mn
Recalling that $W = \cup\{W_i:i \in {\cU}_1\} \in I$, clearly $W_*
\nsubseteq W$, hence we can choose $\delta(*)$ such that
\mn
\begin{enumerate}
\item[$(*)_8$]  $\delta(*) \in W_* \backslash W$.
\end{enumerate}

Now
\begin{enumerate}
\item[$(*)_9$]  if $\alpha \in C_{\delta(*)}$, then 
${\cU}^1_{\delta(*),\alpha} := \{i < i^*:h_\alpha(i) \le_{J_i} 
g_{t_{\delta(*)}}(i)\} = i^*$ mod $I_*$.
\end{enumerate}
\mn
[Why?  Because for each $\alpha \in C_{\delta(*)}$ 
in the partial order $J_* := ( \prod\limits_{i<i^*} J_i,\le_{I_*})$,
we have $h_\alpha \le_{I_*} g_{s_\alpha}$ by the choice of $s_\alpha$, 
i.e. by $(*)_6$ so $\{i < i^*:h_\alpha(i) \le g_{s_\alpha}(i)\} = i^* 
\mod I_*$.  Also $s_\alpha \le_J
t_{\delta(*)}$ because $\delta(*) \in W_*$ see $(*)_7 + (*)_8$ hence
$g_{s_\alpha} \le_{I_*} g_{t_{\delta(*)}}$ by
clause $(f)(\varepsilon)$ of the assumption and by clause $(f)(\delta)$
this means that $\{i < i^*:g_{s_\alpha}(i) \le_{J_i} 
g_{t_{\delta(*)}}(i)\} = i^*$ mod $I_*$.  Together, the last two
sentences give $(*)_9$.]

The proof now splits according to the relevant part of clause (g) of
the assumption.
\bigskip

\noindent
\underline{Case 1}:  \underline{Pos A}:

Now $I_*$ is $\kappa^+$-complete (or just
$|C_{\delta(*)}|^+$-complete), by clause $(\alpha)$ 
(or clause $(\alpha)'$) of Pos$(A)$ so necessarily (by $(*)_9$).
\mn
\begin{enumerate}
\item[$(*)_{10}$]  ${\cU}_* := \bigcap\limits_{\alpha \in
C_{\delta(*)}} {\cU}^1_{\delta(*),\alpha} = i^* \mod I_*$.
\end{enumerate}
\mn
Now if $i \in {\cU}_* \backslash {\cU}_0$, then recall that $f_i:B
\rightarrow J_i$ exemplifies that $\bar s^i$ is not a witness for
$\OBB_{J_i}(\bar C,I)$ and $W_i$ is well-defined and a subset of $W$, 
hence $\delta(*) \notin W_i$ and $\delta(*) \notin W$ by its
choice; hence for some $\alpha_i \in C_\delta$, we
have $\neg(f_i(\alpha_i) \le_{J_i} s^i_\delta)$. But this means that
$i \notin {\cU}^1_{\delta(*),\alpha_i}$ which by $(*)_{10}$ implies
$i \notin {\cU}_*$, a contradiction.  We conclude
\mn
\begin{enumerate}
\item[$(*)_{11}$]  ${\cU}_* \subseteq {\cU}_0$.
\end{enumerate}
\mn
So by $(*)_{10} + (*)_{11}$, we have proved $(*)_2$ which, as noted
above, is sufficient for proving the Lemma \ref{j.45} when Pos(A)
hold; this even gives more.
\bigskip

\noindent
\underline{Case 2}:  Pos$(B)$

\noindent 
Why?  Note that by clauses $(\alpha),(\beta)$ of Pos$(B)$, clause (c)
of $(*)_5$ apply.  So for each $i < i^*$, by $(*)_5(c)$ the sequence
$\langle h_\alpha(i):\alpha \in C_{\delta(*)} \rangle$ is
$\le_{J_i}$-increasing.  Hence by $(*)_9$ the sequence $\langle
\cU^1_{\delta(*),\alpha}:\alpha \in C_{\delta(*)}\rangle$ is
$\subseteq$-decreasing and so $\langle i^* \backslash
\cU^1_{\delta(*),\alpha}:\alpha \in C_{\delta(*)}\rangle$ is a
$\subseteq$-increasing sequence of members of $I_*$ hence by clause
$(\gamma)$ of Pos$(B)$, the set $\cU_* :=
\cap\{\cU^1_{\delta,\alpha}:\alpha \in C_{\delta(*)}\}$ is $\ne
\emptyset$ mod $I_*$.  Now we repeat the proof in Case 1 after
$(*)_{10}$ finishing the proof of Lemma \ref{j.45} when Pos(B) holds.
\bigskip

\noindent
\underline{Case 3}:  Pos$(C)$

This is easier.
\bigskip

\noindent
\underline{Case 4}:  Pos$(D)$

By clauses $(\alpha),(\beta)$ of Pos$(D)$ the assumption of clause (c)
of $(*)_5$ holds hence its conclusion.  We continue as in Case 2.
\bigskip

\noindent
\underline{Case 5}:  Pos$(E)$

Easy.  
\end{PROOF}

\begin{conclusion}
\label{j.45.2}
Assume that $\mu > \cf(\mu) = \sigma,\mu > \kappa = 
\cf(\kappa) \ne \sigma$, and $J$ is
an ideal on $\sigma$ which is $\sigma$-complete (or just $\sigma >
\kappa \Rightarrow J$ is $\kappa^+$-complete).

\noindent
1) If $\mu < \lambda = \cf(\lambda) < \pp^+_J(\mu),S
\subseteq S^\lambda_\kappa$ is stationary, and $\bar C = \langle
C_\delta:\delta \in S\rangle$ is a strict $(\lambda,\kappa)$-ladder
system, tree-like when $\sigma > \kappa$, \then \, for
unboundedly many regular $\theta < \mu$, we have $\OBB_\theta(\bar C)$.

\noindent
2) Assume that for each regular $\lambda \in
(\mu,\pp^+_J(\mu)),S_\lambda \subseteq S^\lambda_\kappa$ is
stationary and $\bar C^\lambda = \langle C^\lambda_\delta:\delta \in
S_\lambda\rangle$ is a strict $(\lambda,\kappa)$-ladder system
which is tree-like when $\sigma > \kappa$. \Then \, for some
$\mu_0 < \mu$, for every regular $\theta \in (\mu_0,\mu)$, for some
$\lambda$ we have $\mu < \lambda = \cf(\lambda) < \pp^+_J(\mu)$
and $\OBB_\theta(\bar C^\lambda)$.
\end{conclusion}

\begin{PROOF}{\ref{j.45.2}}
1) By the ``No hole Conclusion", 
\cite[Ch.II,2.3,pg.53]{Sh:g} there is
a sequence $\langle \lambda_i:i < \sigma\rangle$ of regular cardinals
such that $\mu = \lim_J\langle \lambda_i:i < \sigma\rangle$ and
$\lambda = \tcf(\prod\limits_{i < \sigma} \lambda_i,<_J)$; 
let $\langle g_\alpha:\alpha < \lambda\rangle$ exemplify this.  We shall
apply \ref{j.45}, Pos$(A)$ if $\sigma > \kappa$, Pos$(B)$ if $\sigma <
\kappa$ with $\sigma,J^{\nst}_\lambda \restriction S,
J^{\bd}_{\lambda_i}$, (for $i < \sigma$),
$J,(\lambda,<),\langle g_\alpha:\alpha < \lambda\rangle$ here standing
for $i^*,I,J_i$ (for $i<i^*$), $I_*,J,\langle g_t:t \in J\rangle$ there.

Note that clause (a) of the assumption of \ref{j.45} says that
$\OBB_\lambda(\bar C,J^{\nst}_\lambda \restriction S)$, it holds by
Observation \ref{j.44}(1); the other assumptions of Lemma
\ref{j.45} are also obvious.  So its conclusion holds,
i.e., $\{i < \sigma:\OBB_{\lambda_i}(\bar C)\}$ belongs to $J^+$.

Therefore, since $\mu = \lim_J\langle \lambda_i:i < \sigma\rangle$,
clearly $\mu = \sup\{\lambda_i:i < \sigma$ and $\OBB_{\lambda_i}
(\bar C)\}$ as required.

\noindent
2) Similarly.
\end{PROOF}

\noindent
Note that useful to combine the following older results with \ref{j.45.2} is:
\begin{observation}
\label{j.45.3}
In \ref{j.45.2}, assuming $\mu < \lambda = \cf(\lambda) < \pp^+_J(\mu)$.

\noindent
1) We can find a stationary $S \subseteq
S^\lambda_\kappa$ such that $S \in \check I[\lambda]$.

\noindent
2) For any stationary $S \subseteq S^\lambda_\kappa$ from $\check
 I[\lambda]$ there are a club $E$ of $\lambda$ and a strict 
$(\lambda,\kappa)$-ladder system $\bar C = \langle C_\delta:\delta \in
S \cap E\rangle$ which is tree-like.
\end{observation}

\begin{PROOF}{\ref{j.45.3}}
1) By \cite[\S1]{Sh:420} there is such $S$.

\noindent
2) By \cite[\S1]{Sh:420} there is a strict $S$-ladder system 
$\bar C = \langle C_\delta:\delta \in S\rangle$ which is tree-like.
\end{PROOF}

\begin{conclusion}
\label{j.45.5}
Assume $\mu \in \bold C_\sigma,\kappa = \cf(\kappa) \in \Reg 
\cap \mu \backslash \{\sigma\}$ and for every $\lambda \in \Reg
\cap (2^\mu)^+ \backslash \mu$ the sequence $\bar C_\lambda :=
\langle C^\lambda_\delta:\delta \in S_\lambda\rangle$ is a
$(\lambda,\kappa)$-ladder system.

Assume $\kappa < \sigma$ \underline{or} $\sigma < \kappa < \mu$ and
each $\bar C_\lambda$ is tree-like.

\noindent
1) For every large enough regular $\theta < \mu$ we have
$\OBB_\theta(\bar C_\lambda)$ for some $\lambda \in \Reg \cap (\mu,2^\mu]$.

\noindent
2) If $\kappa < \sigma$ and $\lambda \in [\mu,2^\mu] \cap \Reg$
\then \, for arbitrarily large $\theta \in \Reg 
\cap \mu$ we have $\OBB_\theta(\bar C_\lambda)$.
\end{conclusion}

\begin{PROOF}{\ref{j.45.5}}  By \ref{j.45.2}.
\end{PROOF}

\begin{question}
\label{j.45.7}
1) On entangled linear orders see \cite{Sh:462}, existence is proved
   in some $\mu^+$, but it remains open whether we can demand $\mu =
   \mu^{\aleph_0}$; is the present work helpful?

\noindent
2) If ${\cP} \subseteq {}^\omega \lambda,\mu \le \lambda
\le 2^\mu$ and $|{\cP}| \le \lambda$, can we partition ${\cP}$ to
``few" $\aleph_\omega$-free sets?  What if we add
$2^\mu = \lambda = 2^{< \lambda} = \cf(\lambda)$?
\end{question}
\bigskip

\noindent
\centerline {$* \qquad * \qquad *$}

\begin{claim}
\label{j.46}
1) Assume
\mn
\begin{enumerate}
\item[$(a)$]  $\OBB^+_J(\bar C,I)$ and\footnote{using
$C_\delta$'s of constant cardinality is a loss but only if $\kappa$ is
a limit cardinal.} $\kappa = \cf(\kappa) > 
|C_\delta|$ for $\delta \in S := S(\bar C)$
\sn
\item[$(b)$]  $I$ is $(2^{|i^*|})^+$-complete
\sn
\item[$(c)$]  $\bar J = \langle J_i:i < i^*\rangle$
\sn
\item[$(d)$]  $J_i$ is a partial order such that
$J_i \models \forall s \exists t(s<t)$
\sn
\item[$(e)$]  $I_*$ is a $\kappa$-complete ideal on $i^*$
\sn
\item[$(f)$]  $(\alpha) \quad J$ is a partial order
\sn
\item[${{}}$]  $(\beta) \quad \bar g = \langle g_t:t \in J\rangle$
\sn
\item[${{}}$]  $(\gamma) \quad g_t \in \prod\limits_{i<i^*} J_i$
\sn
\item[${{}}$]  $(\delta) \quad s <_J t \Rightarrow g_s <_{I_*} g_t$
\sn
\item[$(g)$]   if $g^\varepsilon \in  \prod\limits_{i<i^*}
J_i$ for $\varepsilon < \varepsilon^* < \kappa$,
\then \, for some $A \in I^+_*$ and for each
$\varepsilon < \varepsilon^*$, for some $t \in J$,
we have $g^\varepsilon \restriction 
A <_{I_* \restriction A} g_t \restriction A$.
\end{enumerate}
\mn
\Then \, $\{i < i^*:\OBB_{J_i}(\bar C,I)\} \ne \emptyset \mod I_*$.  

\noindent
2) If clause (a) above holds for $I_* \rest A_0$ for any $A_0 \in
   I^+_*$, \then \, we can strengthen the conclusion to $\{i < i^*:
\OBB_{J_i}(\bar C,I)\} = i^* \mod I_*$.
\end{claim}

\begin{PROOF}{\ref{j.46}}
1) We start to repeat the proof of \ref{j.45}.
Let $\bar t = \langle t_\delta:\delta \in S\rangle$ witness
$\OBB^+_J(\bar C,I)$.

Let ${\cU}_1 = \{i<i^*:\bar s^i := \langle
g_{t_\delta}(i):\delta \in S\rangle \in {}^S(J_i)$ is not a witness 
for $\OBB_{J_i}(\bar C,I)\}$ and for $i <i(*)$ let $f_i:
\Dom(\bar C) \rightarrow J_i$ exemplify $\neg \OBB_{J_i}(\bar C,I)$ 
if $i \in {\cU}_1$; hence clearly $\langle f_i(\alpha):i < i^*\rangle \in
 \prod\limits_{i<i^*} J_i$ for $\alpha \in \Dom(\bar C)$.  So $i \in
\cU_1 \Rightarrow W^i := \{\delta \in S:(\exists \alpha \in C_\delta)
(f_i(\alpha) \nleq g_{t_\delta}(i)\} = S \mod I$.

Toward contradiction assume that ${\cU}_1 = i^* \mod I_*$.

For each $A \in I^+_*$ we choose $h_A:\Dom(\bar C) \rightarrow J$ such that:
\mn
\begin{enumerate}
\item[$(*)_1$]  for $\alpha \in \Dom(\bar C)$, if 
there is $t \in J$ such that $\langle f_i(\alpha):i \in A \rangle <_{I_*
\restriction A} (g_t \rest A)$, then 
$\langle f_i(\alpha):i \in A \rangle <_{I_* \restriction A} g_{h_A(\alpha)}$.
\end{enumerate}
\mn
For each $A \in I^+_*$, by the choice of $\bar t$
we know that $W_A := \{\delta \in S$ :
for every $\alpha \in C_\delta$, $h_A(\alpha) \le_J t_\delta\} =S \mod I$.

But the ideal $I$ is $(2^{|i^*|})^+$-complete by clause (b) of the
assumption, hence $W := 
\cap\{W_A:A \in I^+_*\} \cap \{W^i:i \in {\cU}_1\} =S \mod I$.
Now for any $\delta \in W$, consider $\left <\langle
f_i(\alpha):i<i^*\rangle:\alpha \in C_\delta \right>$; it is a
sequence of $< \kappa$ members of ${}^{i^*}J$; hence by clause (f)
for some $A \in I^+_*$, each $\langle f_i(\alpha):i<i^*\rangle 
\restriction A$ has a bound belonging to the set 
$\{g_t \restriction A:t \in J\}$ in 
$(\prod\limits_{i \in A} J_i,<_{I_* \restriction A})$.

So by $(*)_1$ we have 
$\alpha \in C_\delta \Rightarrow \langle f_i(\alpha):i \in A\rangle
<_{I_* \restriction A} g_{h_A(\alpha)}$. But $\delta \in W \subseteq
W_A$, hence $h_A(\alpha) \le_J t_\delta$, which implies that 
$g_{h_A(\alpha)} \le_{I_*} g_{t_\delta}$.  So $\alpha \in 
C_\delta \Rightarrow \langle f_i(\alpha):i \in A \rangle 
\le_{I_* \restriction A} g_{h_A(\alpha)}
<_{I_*} g_{t_\delta}$.  However, $I_*$ is $\kappa$-complete by clause
(e) and $|C_\delta| < \kappa$, hence $B := \{i \in A:(\forall \alpha \in
C_\delta)(f_i(\alpha) <_{J_i} g_{t_\delta}(i)\} = A \mod I_*$.

As ${\cU}_1 = i^* \mod I_*$ and $A \in I^+_*$ clearly ${\cU}_1 \cap A \in
I^+_*$, so ${\cU}_1 \cap A \ne \emptyset$.  But letting 
$i \in A \cap {\cU}_1$, we get a contradiction to the 
assumption that $f_i$ exemplifies $\neg \OBB_{J_i}(\bar C,I)$, 
more exactly $W$ is thin enough, i.e., $W \subseteq W^i$ for $i \in \cU_1$.

\noindent
2) Easy. 
\end{PROOF}

\begin{claim}
\label{1f.55}
If $\lambda_n = \cf(2^{\mu_n}),
2^{\mu_n} < \mu_{n+1}$ for $n < \omega$ and
$(\forall \alpha < 2^{\mu_n})(|\alpha|^{\aleph_1} < 2^{\mu_n}),\chi < \mu =
\Sigma\{\mu_n:n < \omega\} \in \bold C_{\aleph_0}$ and $\lambda \in
\pcf\{\lambda_n:n < \omega\} \backslash \mu)$ (maybe $J$ is
an $\aleph_1$-complete ideal on $\kappa$) \then \, 
$\BB(\lambda,\aleph_{\omega +1},\chi,J_{\omega_1 * \omega})$ recalling
\cite[Definition 0.5=0p.14]{Sh:898}.

\noindent
2) Similarly replacing $\omega$ by $\sigma$, so $\aleph_1$ by $\sigma^+$.
\end{claim}

\begin{remark}
\label{1f.55d}
1) On $J_{\omega_1 * \omega}$ see \cite[0.3=0p.6(3)]{Sh:898}.

\noindent
2)  If $\{\mu:(\forall \alpha < 2^\mu)(|\alpha|^{\aleph_1} < 2^\mu)\}$
 is uncountable the claim apply, i.e. its assumption holds for some
 $\langle (\mu_n,\lambda_n):n < \omega\rangle,\mu,\lambda$.
\end{remark} 

\begin{PROOF}{\ref{1f.55}}
1) By part (2) using $\sigma = \aleph_0$.

\noindent
2) Note
\mn
\begin{enumerate}
\item[$(*)_1$]   without loss of generality $\chi = \chi^{\sigma^+} <
\mu_i$ for every $i < \sigma$.
\end{enumerate}
\mn
[Why?  First $\mu \in \bold C_\sigma$, clearly $\sigma < \mu$ and
$\chi < \mu \Rightarrow 2^\chi < \mu \Rightarrow \chi^{\sigma^+} <
\mu$; so we can replace $\chi$ by $\chi^{\sigma^+}$.  Second, let
$i(*) = \text{ min}\{i:\mu_i > \chi\}$ and replace $\langle \mu_i:i
<\sigma\rangle$ by $\langle \mu_{i(*)+i}:i <\sigma\rangle$.]
\mn
\begin{enumerate}
\item[$(*)_2$]  let $\langle f^i_\alpha:\alpha < 2^{\mu_i}\rangle$
list ${}^{(\mu_i)}\chi$
\sn
\item[$(*)_3$]  choose a sequence $\bar g_i = \langle g^i_\beta:\beta
< 2^{\mu_i}\rangle$ of members of ${}^{\mu_i}({}^\sigma \chi)$ 
such that: if $\beta < \lambda_i$ and $\gamma_j < \beta$ for 
$j < \sigma$ \underline{then} for some $\varepsilon < \mu_i$ we
have $g'_\beta(\varepsilon) = \langle f^i_{\gamma_j}(\varepsilon):j <
\sigma\rangle$.
\end{enumerate}
\mn
[Why?  As $\alpha < 2^{\mu_i} \Rightarrow |\alpha|^\sigma <
2^{\mu_i}$ and $\chi = \chi^\sigma$ by renaming it suffices to
prove: if $\cF \subseteq {}^{(\mu_i)}\chi$ has cardinality $<
2^{\mu_i}$ then for some $g \in {}^{\mu_i}\chi$ we have $(\forall f
\in \cF)(\exists \varepsilon < \mu_i)(f(i) = g(i))$.

As $\mu_i \in \bold C_\sigma$ this is as in \cite{Sh:775} or
\S2 here; that is we can find $\bar f = \langle f_\varepsilon:\varepsilon <
\mu_i\rangle$ exemplifying {\rm Sep}$(\mu_i,\chi,(2^\chi)^+)$ which holds by
\cite[2.6=d.6(d)]{Sh:898}.
Now choose $g \in {}^{(\mu_i)} \chi \backslash \cup\{\Sol_\varrho:
\varrho \in \cF\}$, where {\rm Sol}$_\varrho :=
\{\nu \in \chi$ : if $\varepsilon < \mu_i$ then $\varrho(\varepsilon) =
f_\varepsilon(\nu)\}$ as in $(*)_0$ in the proof of \cite[2.6=d.6]{Sh:898}; 
there is such $g$ as $|\Sol_\varrho| < (2^\chi)^+$ 
for every $\varrho \in {}^{(\mu_i)}\chi$.]
\mn
\begin{enumerate}
\item[$(*)_4$]  without loss of generality $\lambda := 
\tcf(\prod\limits_{i < \sigma} \lambda_i,
<_{J^{\bd}_\sigma})$ is well defined.
\end{enumerate}
\mn
[Why?  By the pcf theorem there is an unbounded $u \subseteq \sigma$
such that $\tcf(\prod\limits_{i \in u} \lambda_i,<_{J^{\bd}_u})$
is well defined; now rename.]
\mn
\begin{enumerate}
\item[$(*)_5$]  if $\lambda < 2^\mu$ then we get the conclusion.
\end{enumerate}
\mn
[Why?  By Definition \cite[1.1=1.3.1]{Sh:898} and
  \cite[1.3=1.3.3(c)]{Sh:898} there is a
$\mu^+$-free $\cF \subseteq {}^\sigma \mu$ of cardinality $\lambda$.
Hence by \cite[1.8 = 1f.13(2A)]{Sh:898} for
any stationary $S \subseteq S^\lambda_{\sigma^+}$ we can
choose an $S$-ladder system $\langle C_\delta:\delta \in S\rangle$
which is $(\mu^+,J_{\sigma^+ * \sigma})$-free.
So by \ref{j.45}, $\OBB_{\lambda_i}(\bar C_1)$ holds for every
$i < \sigma$ large enough recalling Definition \ref{j.35}(2).  
By $(*)_3$ and Theorem \ref{j71} below we can get the
conclusion.]
\mn  
\begin{enumerate}
\item[$(*)_6$]  if $\lambda = 2^\mu$ then we get the conclusion.
\end{enumerate}
\mn
[Why?  Now by \cite[1.26=2b.111]{Sh:898} we can find stationary $S \subseteq
S^\lambda_{\sigma^+}$ and a $(\sigma^{+(\sigma +1)},
J_{\sigma^+ * \sigma})$-free $S$-ladder system.  
We finish similarly to $(*)_5$ but 
here we use $\alpha < 2^{\mu_i} \Rightarrow |\alpha|^{\sigma^+} < 2^{\mu_i}$.]
\end{PROOF}

\begin{theorem}
\label{j71}
1) We have $\BB(\lambda,\bar C,\theta,\kappa)$ \when \,:
\mn
\begin{enumerate}
\item[$(a)$]  $\OBB_\chi(\bar C,I),I$ is $\mu^+$-complete
\sn
\item[$(b)$]  $\kappa(\bar C) = \kappa$ 
\sn
\item[$(c)$]  $\chi = \cf(2^\mu)$ and $\theta < \mu$
\sn
\item[$(d)$]  $\alpha < 2^\mu \Rightarrow |\alpha|^\kappa < 2^\mu$
\sn
\item[$(e)$]  $\Sep(\mu,\theta)$, see Definition \cite[2.1=d.8(2)]{Sh:898}.
\end{enumerate}
\mn
2) Moreover in part (1) we get $\BB(\lambda,\bar C,(2^\mu,\theta),\kappa)$.

\noindent
3) Similarly replacing (d) by
\mn
\begin{enumerate}
\item[$(d)'$]  $\bar C$ is tree-like and $\alpha < 2^\mu \Rightarrow
|\alpha|^{<\kappa>_{\text{\rm tr}}} < 2^\mu$.
\end{enumerate}
\end{theorem}

\begin{PROOF}{\ref{j71}}  
1) By (2).

\noindent
2),3)  Let $\bar t = \langle t_\delta:\delta \in S(\bar C) \rangle
\in S(\bar C)_\theta$ witness OBB$_\chi(\bar C)$.

We now repeat the proof of \cite[1.10]{Sh:775} or of
\cite[2.2=d.6]{Sh:898} here 
using $t_\delta$ instead of $\delta$ for $\delta \in S$.
\end{PROOF}

\begin{discussion}
\label{j74}
On \ref{j71}:  

\noindent
1) We use the obvious decomposition $\langle 
{\cF}_\alpha:\alpha < \chi\rangle$ of ${}^\mu \theta$.  There may be others.

\noindent
2) We may replace ``$\chi = \text{\rm cf}(2^\mu)$" by $\chi = \lambda
   = \text{\rm min}\{\lambda:2^\lambda > 2^\mu\}$ as in \S2.

\noindent
3) We may phrase the condition on $\bold F:{}^\mu \theta \rightarrow
   \Upsilon$ not only when $\Upsilon = \text{\rm cf}(\theta^\mu)$.

\noindent
4) Like (3) but for our specific problem: Hom$(G,\bbZ) = \{0\}$.
\end{discussion}
\bigskip

\centerline {$* \qquad * \qquad *$}
\bigskip

We look at another way to get cases of OBB.  Recall the definition of
$\gd^{< \sigma}_{\theta,\gd^\sigma_\theta}$ from Matet-Roslanowski-Shelah
\cite[1.1]{MRSh:799} which proves this number can; i.e. consistently, have
cofinality $\aleph_0$ and be $< \gd_\theta$; not used in the rest of
the paper.
\begin{definition}
\label{j.63}
Assume $\mu \ge \theta = \text{ cf}(\theta) \ge \sigma$ and $I$ an
   ideal on $\mu$ (so by our notation determine $\mu$).

\noindent
1) Let $\gd_{I,\theta,<\sigma}$ be
$\text{\rm min}\{|{\cF}|:{\cF} \subseteq {}^\mu \theta \text{
  has no } (I,\sigma)\text{-bound}\}$ where

\noindent
2) We say that $g \in {}^\mu \theta$ is a $(I,\sigma)$-bound of
${\cF} \subseteq {}^\mu \theta$ \when \, for any ${\cF}' \subseteq 
{\cF} \text{ of cardinality } < \sigma$ for $I^+$-many
$\varepsilon < \mu$ we have $(\forall f \in {\cF}')
(f(\varepsilon) < g(\varepsilon))$.

\noindent
3) Let $\gd^{\text{\rm inc}}_{I,\theta,\sigma}$ be defined similarly
but ${\cF}' = \{f_i:i < \sigma\}$ with the sequence $\langle f_i:i < 
\sigma\rangle$ being $<_{J^{\text{\rm bd}}_\kappa}$-increasing.

\noindent
4) Let $\gD_{I,\theta,\sigma}$ be the set of regular $\chi > \theta$
such that there is a $\subseteq$-increasing sequence
   $\langle {\cF}_\varepsilon:\varepsilon < \chi\rangle$ of subsets
of ${}^\mu \theta$ such that: ${\cF}_\varepsilon$ has a 
$(I,< \sigma)$-bound for $\varepsilon <  \chi$ and 
$\cup\{{\cF}_\varepsilon:\varepsilon < \chi\} = {}^\mu \theta$. 

\noindent
5) Let $\gD^{\text{\rm seq}}_{I,\theta,\sigma} = 
\text{\rm  min}\{|{\cF}|:{\cF} \subseteq \cup\{{}^\iota({}^\mu
   \theta):\iota < \sigma\}$ has no $(I,< \sigma)$-bound $g\}$ where

\noindent
6) $g$ is a $(I,<\sigma)$-bound of ${\cF} \subseteq
   \cup\{{}^\iota({}^\mu \theta):\iota < \sigma\}$ \when \, $g
 \in {}^\mu \theta$ and if $\bar g' = \langle g_\iota:\iota < \sigma
\rangle$ and $\iota < \sigma \Rightarrow \bar g' \rest \iota
   \in {\cF}$ then for $I^+$-many $\varepsilon \le \mu$ we have
$\iota < \sigma \Rightarrow g_\iota(\varepsilon) \le g(\varepsilon)$.

\noindent
7) If $I$ is the ideal $\{\emptyset\}$ on $\mu$ then we may write
   $\mu$ instead of $I$; omitting $\mu$ (and $I$) means $\mu
   = \theta \wedge I = J^{\text{\rm bd}}_\theta$.
\end{definition}

\begin{definition}
\label{j.64}
We define $\gd^{\text{\rm eq}}_{I,\theta,\sigma}$, etc. when $\le$ is
replaced by $=$.
\end{definition}

\begin{claim}
\label{j.67}
1) Assume $\theta > \kappa$ are regular cardinals and $\lambda = \text{\rm
   cf}(\gd_{\theta,\kappa}) > \theta$.  If $\bar C$ is a
$(\lambda,\kappa)$-ladder system \then \, {\rm OBB}$_\theta(\bar C)$
   (and $\lambda > \theta$).

\noindent
2) Assume $\theta > \kappa$ are regular and $\lambda = 
\text{\rm cf}(\gd^{\text{\rm seq}}_{\theta,\kappa}) > 
\theta$.  If $\bar C$ is a tree-like
$(\lambda,\kappa)$-ladder system \underline{then} {\rm OBB}$_\theta(\bar C)$.

\noindent
3) Assume $\theta > \kappa,\lambda \ge \chi$ are regular cardinals and
   $\chi \in \gD_{\theta,\kappa}$.  If {\rm OBB}$_\chi(\bar C,I),I$ is
   $\theta^+$-complete, $\lambda(\bar C) = \lambda$ 
and $\kappa(\bar C) = \kappa$ \then \, {\rm OBB}$_\theta(\bar C)$.

\noindent
4) If $\theta > \sigma$ are regular cardinals \then \,
{\rm cf}$(\gd_{\theta,\sigma}) > \theta \Rightarrow \text{\rm cf}
(\gd_{\theta,\sigma}) \in \gD_{\theta,\sigma}$ and
{\rm cf}$(\gd_{\theta,< \sigma}) 
\notin [\sigma,\theta]$ and {\rm cf}$(\gd^{\text{\rm
inc}}_{\theta,\sigma}) \notin \theta^+ \cap \text{\rm Reg} \backslash
 \{\sigma\}$.
\end{claim}

\begin{remark}  1) Assume $\theta_0 = 2^{\aleph_0} = \text{\rm
   cf}(2^{\aleph_0}),\theta_{n+1} = \gd_{\theta_n,\aleph_1}$.  Then we
   can use part (3) \underline{but} $\mu = \sup\{\theta_n:n <
   \omega\}$ is not necessarily strong limit.  We can consider
   $\langle \theta_i:i < \omega_1\rangle$ getting $\kappa(\bar C) =
   \aleph_1$.  Do we get $\{\delta:\text{pp}(\sum \limits_{i < \delta}
   \theta_i)=\text{ cov} (\sum\limits_{i < \delta}
   \theta_i,\aleph_1,\aleph_1,2)\} \in {\cD}_{\omega_1}$?

\noindent
2) By \cite{MRSh:799} we know that it is consistent with ZFC that:
   there is $\cA \subseteq \text{ NS}_\theta =\{A \subseteq \theta:A$
   not stationary$\}$ of cardinality $< \text{ cf}(\text{NS}_\theta)$
   such that every $A \in \text{ NS}_\theta$ is included in the union
   of $< \theta$ of them.
\end{remark}

\begin{PROOF}{\ref{j.67}}
1) By \ref{j.44}(1) we have OBB$_\lambda(\bar C)$, by part
   (4), {\rm cf}$(\gd_{\theta,\kappa}) \in \gD_{\theta,\kappa}$ and
by part (3) with $\chi := \lambda$ we deduce OBB$_\theta(\bar C)$ as
promised.

\noindent
2) Similarly.

\noindent
3) Clearly
we can find $\bar{\cF}$ such that:
\mn
\begin{enumerate}
\item[$(*)_2$]  $(a) \quad \bar{\cF} = \langle 
{\cF}_\varepsilon:\alpha \le \chi\rangle$ is $\subseteq$-increasing
continuous
\sn
\item[${{}}$]  $(b) \quad \bar g = \langle g_\varepsilon:\varepsilon <
\chi \rangle,g_\varepsilon \in {}^\theta \theta$
\sn
\item[${{}}$]  $(c) \quad g_\varepsilon$ is a $(< \kappa)$-bound of
${\cF}_\varepsilon$ for $\alpha < \lambda$
\sn
\item[${{}}$]  $(d) \quad {\cF}_\chi = {}^\theta \theta$.
\end{enumerate}
\mn
As we are assuing OBB$_\chi(\bar C)$ let $\bar t$ be such that:
\mn
\begin{enumerate}
\item[$(*)_3$]  $\bar t = \langle t_\alpha:\alpha < \lambda\rangle \in
{}^\lambda \chi$ is a witness for OBB$_\chi(\bar C)$.
\end{enumerate}
\mn
For each $i < \theta$ let:
\mn
\begin{enumerate}
\item[$(*)_4$]  $\bar t_i = \langle t^i_\delta:\delta \in S\rangle \in
{}^\lambda \theta$ be $t^i_\delta = g_{t_\delta}(i)$.
\end{enumerate}
\mn
Let
\mn
\begin{enumerate}
\item[$(*)_5$]  ${\cU}_1 := \{i < \theta:\bar t_i$ is not a witness
for OBB$_\theta(\bar C)\}$.
\end{enumerate}
\mn
For each $i < \theta$ let $(f_i,E_i)$ be such that
\mn
\begin{enumerate}
\item[$(*)_6$]  $(a) \quad f_i:\lambda \rightarrow \theta$
\sn
\item[${{}}$]   $(b) \quad E_i = S$ mod $I$
\sn
\item[${{}}$]   $(c) \quad$ if $i \in {\cU}_1$ and $\delta \in E_i$
then $(\exists \alpha \in C_\delta)(f_i(\alpha) \ge t^i_\delta)$ that
is

\hskip25pt  $(\exists \alpha \in C_\delta)(f_i(\alpha) \ge g_{t_\delta}(i))$.
\end{enumerate}
\mn
Let $E = \cap\{E_i:i < \theta\}$ hence $S \backslash E \in I$.  
For each $\alpha < \lambda$ let
$h_\alpha:\theta \rightarrow \theta$ be $h_\alpha(i) = f_i(\alpha)$
and $s_\alpha = \text{ min}\{\varepsilon < \chi:h_\alpha \in 
{\cF}_\varepsilon\}$ so $\alpha \mapsto s_\alpha$ is a function from
$\lambda$ to $\theta$.  By the choice of $\bar t$ we have

\[
{\cW} := \{\delta \in S:(\forall \alpha \in C_\delta)(s_\alpha \le
t_\delta)\} \ne \emptyset \text{ mod } I.
\]

\mn
So we can choose $\delta(*) \in {\cW} \cap E$.  Now
$\{h_\alpha:\alpha \in C_{\delta(*)}\} \subseteq {\cF}_{\delta(*)}$
hence by the choice of g$_{t_{\delta(*)}}$ we have
\mn
\begin{enumerate}
\item[$(*)_7$]  ${\cU}_2 = \{i < \theta:(\forall \alpha \in
C_\delta)(h_\alpha(i) \le g_{t_\delta}(i))$ equivalently $(\forall
\alpha \in C_\delta)(f_i(\alpha) \le t^i_\delta)\} \ne \emptyset$ mod
$J^{\text{\rm bd}}_\theta$.
\end{enumerate}
\mn
Now
\mn
\begin{enumerate}
\item[$(*)_8$]  ${\cU}_1 \cap {\cU}_2 = \emptyset$
\end{enumerate}
\mn
hence
\mn
\begin{enumerate}
\item[$(*)_9$]  ${\cU}_1 \ne \theta$ mod $J^{\text{\rm bd}}_\theta$
\end{enumerate}
\mn
hence
\mn
\begin{enumerate}
\item[$(*)_{10}$]  $\theta \backslash 
{\cU}_1 \ne \emptyset$, i.e. for some $i < \theta,\bar t_i$
witness OBB$_\theta(\bar C)$.
\end{enumerate}
\mn
This is enough.

\noindent
4) First, we have to show that $\chi := \text{\rm cf}(\gd_{\theta,\sigma})
\in \gD_{\theta,\kappa}$ if cf$(\gd_{\theta,\kappa}) > \theta$.

So let ${\cF} \subseteq {}^\theta \theta$ be of cardinality
$\gd_{\theta,\kappa}$ such that no $g \in {}^\theta \theta$ is a $(\le
\kappa)$-bound of ${\cF}$.  Let $\langle {\cF}_\varepsilon:
\varepsilon < \chi\rangle$ be $\subseteq$-increasing
with union ${\cF}$ such that $\varepsilon < \chi \Rightarrow 
|{\cF}_\varepsilon| < |{\cF}| = \gd_{\theta,\kappa}$.  Let 
${\cF}^+_\varepsilon = \{g \in {}^\theta \theta$: for some ${\cF}'
\subseteq {\cF}$ of cardinality $< \sigma$ and $i_* < \theta$ we have $i_*
< i < \theta \Rightarrow g(i) = \sup\{f(i):f \in {\cF}'\}\}$.
\mn
\begin{enumerate}
\item[$\odot_1$]   $\langle {\cF}^+_\varepsilon:\varepsilon <
\chi\rangle$ is $\subseteq$-increasing.
\end{enumerate}
\mn
Now we shall prove
\mn
\begin{enumerate}
\item[$\odot_2$]   $\cup\{{\cF}^+_\varepsilon:\varepsilon < \chi\}
= {}^\theta \theta$.
\end{enumerate}
\mn
Let $h \in {}^\theta \theta$, by the choice of ${\cF}$ the function
$h$ is not a $(< \sigma)$-bound of ${\cF}$ hence there is
${\cF}' \subseteq {\cF}$ of cardinality $< \sigma$ witnessing it
which means that for every large enough $i < \theta,
h(i) \le \sup\{f(i):f \in {\cF}'\}$.  
Let ${\cF}' = \{f_j:j < j_* < \sigma\}$, let
$\varepsilon_h(j) = \text{ min}\{\varepsilon:f_j \in 
{\cF}_\varepsilon\}$ and let $\varepsilon_h := 
\sup\{\varepsilon_h(j):j < j_*\} < \chi$ so clearly 
$h \in {\cF}^+_{\varepsilon_h}$ hence $\odot_2$ is proved.

Lastly, obviously
\mn
\begin{enumerate}
\item[$\odot_3$]   $g_\varepsilon$ is a $(< \sigma)$-bound of 
${\cF}^+_\varepsilon$.
\end{enumerate}
\mn
Together we have shown that $\chi \in \gD_{\theta,\sigma}$.  Second, 
why $\chi = \text{ cf}(\gd_{\theta,\sigma}) \notin [\sigma,\theta]$?

Let $\langle {\cF}_\varepsilon:\varepsilon < \chi\rangle,\langle
g_\varepsilon:\varepsilon < \chi\rangle$ be as above.
 Let $g \in {}^\theta \theta$ be a common $<_{J^{\text{\rm
bd}}_\theta}$-upper bound of $\{g_\varepsilon:\varepsilon < \chi\}$,
exist as $\chi \le \theta$
\mn
\begin{enumerate}
\item[$\boxplus$]  $g$ is a $(< \sigma)$-bound of ${\cF}$.
\end{enumerate}
\mn
[Why?  If ${\cF}' \subseteq {\cF}$ has cardinality $< \sigma$
then ${\cF}' \subseteq {\cF}_\varepsilon$ for some $\varepsilon
< \chi$, so $\{i:(\forall f \in {\cF}')(f(i) \le g_\varepsilon(i)\}
\ne \emptyset$ mod $J$ but $g_\varepsilon \le_{J^{\text{\rm
bd}}_\theta} g$, so $g$ indeed is a $(< \sigma)$-bound of ${\cF}$.]

But $\boxplus$ contradicts the choice of ${\cF}$.
\end{PROOF}

\begin{remark}
\label{j68}
1) We may combine OBB and the results of \S2 (or \cite{Sh:775}),
   e.g. see also  \ref{j.49} below.

\noindent
2) Like \ref{j.67} for $\gd^{\text{\rm eq}}_{\theta,\sigma}$,
etc. (connection to {\rm Sep}).
\end{remark}

\centerline {$* \qquad * \qquad *$}

\noindent
We look at relatives of OBB, though we shall not use them.

\noindent
Other variants are
\begin{definition}
\label{k38}
1)  OBB$^0_J(\bar C,I)$ is defined as in \ref{j.35}(1) but we 
demand only $\{\delta \in S:\delta = \sup\{\alpha \in C_\delta:
f(\alpha) \le_J t_\delta\}\} \ne \emptyset$ mod $I$.

\noindent
1A) Assume otp$(C_\delta) = \xi$ for every $\delta \in S$ and $J_*$ is an
ideal on $\xi$.  We define OBB$^0_{J,J_*}(\bar C,I)$ as in Definition
\ref{j.35}, but we demand only that $\{\delta \in S:\text{ otp}(\alpha
\cap C_\delta):\alpha \in C_\delta$ and $f(\alpha) \le_J t_\delta\}
\ne \emptyset$ mod $J_*\} \ne \emptyset$ mod $I$.
\newline
2) Let OBB$^1_J(\bar C,I)$ mean OBB$_J(\bar C,I)$.

\noindent
2A) Under the assumption of part (1A), we define OBB$^0_{J,J_*}(\bar
C,I)$ as in Definition \ref{j.35} but we demand $\{\delta \in
S:\{\text{otp}(\alpha \cap C_\delta):\alpha \in C_\delta$ and
$f(\alpha) \le_J t_\delta\} = \xi$ mod $J_*\} \ne \emptyset$ mod $I$.

\noindent
3) OBB$^2_J(\bar C,I)$ when
\mn
\begin{enumerate}
\item[$(a)$]  $J$ is a partial order
\sn
\item[$(b)$]  $I$ is an ideal on $S = S(\bar C)$
\sn
\item[$(c)$]  there is $\bar t$ which is a witness for OBB$^2_J(\bar C,I)$
which means:
\sn
\begin{enumerate}
\item[$(\alpha)$]   $\bar t = \langle t_\delta:\delta \in
S\rangle$
\sn
\item[$(\beta)$]    $t_\delta \in J$
\sn
\item[$(\gamma)$]   if $f:\cup\{C_\delta:\delta \in S\}
\rightarrow J$, \underline{then} the set $\{\delta \in S:(\forall
\alpha \in C_\delta)(t_\delta \nleq_J f(\alpha))\}$ belongs
to $I^+$.
\end{enumerate}
\end{enumerate}
\mn
4) OBB$^3_J(\bar C,I)$ is defined similarly replacing $(c)(\gamma)$ by
\mn
\begin{enumerate}
\item[$(c)(\gamma)'$]  if $f:\cup\{C_\delta:\delta \in S\} \rightarrow
J$ \underline{then} the set $\{\delta \in S:(\forall \alpha \in
C_\delta)(t_\delta \nleq_J f(\alpha)\} = S(\bar C)$ mod $I$.
\end{enumerate}
\mn
5) OBB$^4_J(\bar C,I)$ is OBB$^+_J(\bar C,I)$, i.e. 
is defined as in \ref{j.35}(1) but in the end 
$\{\delta \in S:(\forall \alpha \in C_\delta)(f(\alpha) \le_J
 t_\delta)\} = S$ mod $I$.
\end{definition}

\begin{discussion}
\label{k40}
An example of \ref{k42}(1A): $S = S^\lambda_{\aleph_2},J_* =
J_{\aleph_2 + \aleph_1}$ and we use a parallel of \ref{j.45}.
\end{discussion}

\begin{claim}
\label{k42}
0) Assume $J$ is a partial order, $\bar C = \langle C_\delta:\delta
 \in S\rangle,C_\delta \subseteq \delta,\otp(C_\delta) =
 \kappa,\kappa$ a cardinal (or ordinal), $J_*$ an ideal on $\xi,I$
   an ideal on $S$.

\noindent
1) If {\rm OBB}$^0_{J,J_*}(\bar C,I)$ and $J$ is
   $(2^{|\kappa|})^+$-directed, $I$ is $(2^{|\kappa|})^+$-complete
   \then\, for some $A \subseteq \xi,A \ne \emptyset$ {\rm mod} $J_*$
   we have {\rm OBB}$^1_{J,J_*}(\bar C,\bar I)$.

\noindent
2) For some $\cU \in I^+_*$ we have $i \in \cU \Rightarrow \text{\rm
   OBB}^0_{J_i,J_*}(\bar C,I)$ \when \, below (a)-(f) below hold where
\mn
\begin{enumerate}
\item[$(a)$]  {\rm OBB}$^0_{J,J_*}(\bar C,\bar I),S = S(\bar C)$, {\rm
otp}$(C_\delta) = \kappa$ for $\delta \in S$
\sn
\item[$(b)-(e)$]  as in \ref{j.45}
\sn
\item[$(f)$]  $(\alpha) \quad J_i$ is $\sigma$-directed
\sn
\item[${{}}$]  $(\beta) \quad \bar{\cU} = \langle \cU_\alpha:\alpha
\in \text{\rm Dom}(\bar C)\rangle,\cU_\alpha \subseteq
\alpha,|\cU_\alpha| < \sigma$
\sn
\item[${{}}$]  $(\gamma) \quad$ if $\delta \in S,\bar{\cU} = \langle
\cU_\alpha:\alpha \in C_\delta \rangle,\cU_\alpha \in I^+_*$ for
$\alpha \in C_\delta$ and

\hskip25pt  $\alpha \in C_\delta \wedge \beta \in C_\delta \wedge 
\alpha \in \cU_\beta \Rightarrow \cU_\beta \subseteq \cU_\alpha$
\then \, $\cap\{\cU_\alpha:\alpha \in C_\delta\} \in I^+_*$.
\end{enumerate}
\mn
3) Like part (2) but
\mn
\begin{enumerate}
\item[$(a)^+$]  {\rm OBB}$^1_{J.J_*}(\bar C,I),S = S(\bar C)$, {\rm
otp}$(C_\delta) = \kappa$ for $\delta \in S$
\sn
\item[$(f)^-$]  like $(f)$ but in $(\gamma)$ we hvae $\cU_\alpha =
i^*$ {\rm mod} $I_*$.
\end{enumerate}
\end{claim}

\begin{proof}
Like \ref{j.45}.
\end{proof}

\begin{claim}
\label{k44}
1) If $J$ is linearly ordered, then 
{\rm OBB}$^2_J(\bar C,I) \Rightarrow \text{\rm OBB}^1_J(\bar C,I)$.
\newline
2) {\rm OBB}$^1_J(\bar C,\bar I) \Rightarrow 
\text{\rm OBB}^0_J(\bar C,\bar I)$ and
{\rm OBB}$^1_J(\bar C,\bar I) \Rightarrow 
\text{\rm OBB}^2_J(\bar C,\bar I)$ and
{\rm OBB}$^3_J(\bar C,\bar I) \Rightarrow \text{\rm OBB}^2_J(\bar C,\bar I)$.
\end{claim}

\begin{claim}
\label{k47}
Assume that $J_*$ is a $\sigma$-directed
partial order.
\newline1) If $J_* = J^{\text{\rm bd}}_\sigma,\theta = \text{\rm cf}(\theta) >
2^\sigma,\sigma = \sigma^{< \kappa}$ and {\rm OBB}$^3_\theta(\bar
C,I)$, \underline{then} {\rm OBB}$^2_{J_*}(\bar C,I)$.
\newline2) We can use any $\kappa$-directed partial order $J_*$ of cardinality
$\sigma$.
\end{claim}

\begin{claim}
\label{j.49}
In \ref{j.46}(2), assume we are given
$\langle(J^\xi,\bar J^\varepsilon,\bar g^\xi):\xi < \xi(*)\rangle$
such that
\mn
\begin{enumerate}
\item[$(\alpha)$]  $J^\xi,\bar J^\xi = \langle J^\xi_i:i <
i(\xi)\rangle,\bar g^\xi = \langle g^\xi_t:t \in J\rangle,
\bar C,\kappa,I$ are as in \ref{j.46}(2) for each $\xi < \xi(*)$
 \item[$(\beta)$]  $\alpha < \theta \Rightarrow |\alpha|^{|\xi(*)|}
< \theta$.
\end{enumerate}
\mn
\underline{Then} we can find a $\theta$-complete ideal 
$I' \supseteq I$ such that
for each $\xi$, for some $j_\xi < i(\xi)$, {\rm OBB}$^3_{J^\xi_{j_\xi}}(\bar
C,I')$ holds.
\end{claim}

\begin{PROOF}{\ref{j.49}}
Like \ref{j.46}.  
\end{PROOF}

\begin{observation}
\label{j.45t}
{\rm OBB}$^\ell_{J_2}(\bar C,I)$ holds \underline{when}
\mn
\begin{enumerate}
\item[$(a)$]  $J_1,J_2$ are partial orders
\sn
\item[$(b)$]  {\rm OBB}$^\ell_{J_1}(\bar C,I)$
\sn
\item[$(c)$]  $f:J_1 \rightarrow J_2$
\sn
\item[$(d)$]  for every $t^* \in J_2$, for some $s^* \in J_1$, we
have $(\forall s \in J_1)(h(s) <_{J_2} t^* \Rightarrow s <_{J_2} s^*)$.
\end{enumerate}
\end{observation}

\begin{PROOF}{\ref{j.45t}}
Let $\bar t = \langle t_\delta:\delta \in S(\bar C)\rangle$ 
exemplifies OBB$^\ell_{J_1}(\bar C,I)$.

For each $\delta \in S(\bar C)$ let $s_\delta \in J_1$ be such that
$t_\delta \le_J s_\delta$.  It is enough to show that 
$\bar s = \langle s_\delta:\delta \in S(\bar C)\rangle$ 
exemplifies OBB$^\ell_{J_2}(\bar C,I)$.  So assume $f_1:
\text{Dom}(\bar C) \rightarrow J_1$ and let us define $f_2:
\text{Dom}(\bar C) \rightarrow J_2$ by $f_2(\alpha) h(f_1(\alpha))$.

Let $W = \{\delta \in S:(\forall \alpha \in C_\delta)(f_2(\alpha)
\le_{J_2} t_\delta)\}$. So $\ell =2 \Rightarrow W \in I^+$, and 
$\ell 3 \Rightarrow W = S(\bar C)$ mod $I$.  Hence it suffices to show that
$(\forall \alpha \in C_\delta)[f_1(\alpha) \le_{J_1} s_\delta]$, and
hence it is enough to prove:

\[
(\forall s \in J_1)(h(s) \le_{J_2} t_\delta \Rightarrow h(s) \le_{J_1}
s_\delta).
\]
\end{PROOF}

\def\germ{\frak} \def\scr{\cal} \ifx\documentclass\undefinedcs
  \def\bf{\fam\bffam\tenbf}\def\rm{\fam0\tenrm}\fi 
  \def\defaultdefine#1#2{\expandafter\ifx\csname#1\endcsname\relax
  \expandafter\def\csname#1\endcsname{#2}\fi} \defaultdefine{Bbb}{\bf}
  \defaultdefine{frak}{\bf} \defaultdefine{=}{\B} 
  \defaultdefine{mathfrak}{\frak} \defaultdefine{mathbb}{\bf}
  \defaultdefine{mathcal}{\cal}
  \defaultdefine{beth}{BETH}\defaultdefine{cal}{\bf} \def\bbfI{{\Bbb I}}
  \def\mbox{\hbox} \def\text{\hbox} \def\om{\omega} \def\Cal#1{{\bf #1}}
  \def\pcf{pcf} \defaultdefine{cf}{cf} \defaultdefine{reals}{{\Bbb R}}
  \defaultdefine{real}{{\Bbb R}} \def\restriction{{|}} \def\club{CLUB}
  \def\w{\omega} \def\exist{\exists} \def\se{{\germ se}} \def\bb{{\bf b}}
  \def\equivalence{\equiv} \let\lt< \let\gt>
  \def\implies{\Rightarrow}\def\mathfrak{\bf}\def\germ{\frak} \def\scr{\cal}
  \ifx\documentclass\undefinedcs
  \def\bf{\fam\bffam\tenbf}\def\rm{\fam0\tenrm}\fi 
  \def\defaultdefine#1#2{\expandafter\ifx\csname#1\endcsname\relax
  \expandafter\def\csname#1\endcsname{#2}\fi} \defaultdefine{Bbb}{\bf}
  \defaultdefine{frak}{\bf} \defaultdefine{=}{\B} 
  \defaultdefine{mathfrak}{\frak} \defaultdefine{mathbb}{\bf}
  \defaultdefine{mathcal}{\cal}
  \defaultdefine{beth}{BETH}\defaultdefine{cal}{\bf} \def\bbfI{{\Bbb I}}
  \def\mbox{\hbox} \def\text{\hbox} \def\om{\omega} \def\Cal#1{{\bf #1}}
  \def\pcf{pcf} \defaultdefine{cf}{cf} \defaultdefine{reals}{{\Bbb R}}
  \defaultdefine{real}{{\Bbb R}} \def\restriction{{|}} \def\club{CLUB}
  \def\w{\omega} \def\exist{\exists} \def\se{{\germ se}} \def\bb{{\bf b}}
  \def\equivalence{\equiv} \let\lt< \let\gt>
\providecommand{\bysame}{\leavevmode\hbox to3em{\hrulefill}\thinspace}
\providecommand{\MR}{\relax\ifhmode\unskip\space\fi MR }
\providecommand{\MRhref}[2]{%
  \href{http://www.ams.org/mathscinet-getitem?mr=#1}{#2}
}
\providecommand{\href}[2]{#2}


\begin{thebibliography}{}

\bibitem[Sh:g]{Sh:g}
Saharon Shelah, \emph{{Cardinal Arithmetic}}, {Oxford Logic Guides}, vol.~29,
  {Oxford University Press}, 1994.

\bibitem[Sh:420]{Sh:420}
\bysame, \emph{{Advances in Cardinal Arithmetic}}, {Finite and Infinite
  Combinatorics in Sets and Logic}, Kluwer Academic Publishers, 1993, N.W.
  Sauer et al (eds.). 0708.1979, pp.~355--383.

\bibitem[Sh:462]{Sh:462}
\bysame, \emph{{$\sigma $-entangled linear orders and narrowness of products of
  Boolean algebras}}, {Fundamenta Mathematicae} \textbf{153} (1997), 199--275,
  math.LO/9609216.

\bibitem[Sh:775]{Sh:775}
\bysame, \emph{{Middle Diamond}}, Archive for Mathematical Logic \textbf{44}
  (2005), 527--560, math.LO/0212249.

\bibitem[MRSh:799]{MRSh:799}
Pierre Matet, Andrzej Roslanowski, and Saharon Shelah, \emph{{Cofinality of the
  nonstationary ideal}}, Transactions of the American Mathematical Society
  \textbf{357} (2005), 4813--4837, math.LO/0210087.

\bibitem[Sh:898]{Sh:898}
Saharon Shelah, \emph{{pcf and abelian groups}}, {Forum Mathematicum}
  \textbf{accepted (provided changes are made)}, 0710.0157.

\end{thebibliography}

\end{document}